\documentclass[12pt]{article}
\usepackage[a4paper, margin=2.5cm]{geometry}
\pdfoutput=1
\usepackage{amsmath}
\usepackage[utf8]{inputenc}
\usepackage[UKenglish]{babel}

\usepackage[normalem]{ulem} 
\usepackage[dvipsnames]{xcolor}
\usepackage{aliascnt}
\usepackage{amssymb}
\usepackage{amsthm}
\usepackage{faktor}
\usepackage{mathtools}
\usepackage{amsfonts}
\usepackage{enumerate}
\usepackage{enumitem}
\usepackage{multicol}
\usepackage{float}
\usepackage{tikz}
\usepackage{tikz-cd} 
\usepackage{comment}
\usepackage{cite}
\usepackage[round,comma,sort,
]{natbib}
\usepackage{hyperref}
\hypersetup{
    colorlinks=true,       
    linkcolor=Cerulean,          
    citecolor=Thistle,       
    filecolor=Plum,      
    urlcolor=Orchid          
}

\usepackage[nameinlink]{cleveref}

\DeclareMathOperator{\neutre}{\mathbf{e}}
\DeclareMathOperator{\N}{\mathbb{N}}

\newcommand{\conjugate}[2]{{#2}{#1}{#2}^{-1}}

\newcommand{\edge}[1]{\operatorname{E}(#1)}

\newcommand{\pc}[1]{\operatorname{PC}(#1)}

\DeclareMathOperator{\Supp}{\mathrm{Supp}}
\DeclareMathOperator{\Ribb}{\mathrm{Ribb}}
\DeclareMathOperator{\Conj}{\mathrm{Conj}}

\newcommand{\ver}[1]{\operatorname{V}(#1)}

\newcommand{\Addresses}{{
		\bigskip
		\footnotesize

    María Cumplido, \textsc{Instituto de Matemáticas de la Universidad de Sevilla (IMUS) and Departmento de Álgebra, Facultad de Matemáticas, Universidad de Sevilla. Calle Tarfia s/n 41012, Seville, Spain}\par\nopagebreak
            \textit{E-mail address}: \texttt{cumplido@us.es}

    \medskip
    
    Marina Salamero, \textsc{Instituto de Matemáticas de la Universidad de Sevilla (IMUS) and  Departmento de Álgebra, Facultad de Matemáticas, Universidad de Sevilla. Calle Tarfia s/n 41012, Seville, Spain}\par\nopagebreak
            \textit{E-mail address}: \texttt{msalamero@us.es}
            
    \medskip

    Giovanni Sartori, \textsc{Department of Mathematics, Heriot-Watt University and Max\-well Institute for Mathematical Sciences, Edinburgh, UK}\par\nopagebreak
            \textit{E-mail address}: \texttt{gs2057@hw.ac.uk}

    \medskip

            Mireille Soergel, \textsc{Max Planck Institute for Mathematics in the Sciences, Inselstrasse 22, 04103 Leipzig, Germany}\par\nopagebreak
		\textit{E-mail address}: \texttt{soergel@mis.mpg.de}
		
}}

\title{Parabolic subgroups of Dyer groups}
\author{María Cumplido, Marina Salamero, Giovanni Sartori and Mireille Soergel}
\date{\today}

\begin{document}

\maketitle

% ------- Theorem styles -------
\theoremstyle{plain}
\newtheorem{theorem}{Theorem}
\numberwithin{theorem}{section}

\newaliascnt{lemma}{theorem}
\newtheorem{lemma}[lemma]{Lemma}
\aliascntresetthe{lemma}
\providecommand*{\lemmaautorefname}{Lemma}

\newaliascnt{proposition}{theorem}
\newtheorem{proposition}[proposition]{Proposition}
\aliascntresetthe{proposition}
\providecommand*{\propositionautorefname}{Proposition}

\newaliascnt{corollary}{theorem}
\newtheorem{corollary}[corollary]{Corollary}
\aliascntresetthe{corollary}
\providecommand*{\corollaryautorefname}{Corollary}

\newaliascnt{conjecture}{theorem}
\newtheorem{conjecture}[conjecture]{Conjecture}
\aliascntresetthe{conjecture}
\providecommand*{\conjectureautorefname}{Conjecture}

%--- Introduction theorems ---

\newtheorem{thmintro}{Theorem}
\renewcommand*{\thethmintro}{\Alph{thmintro}}

\newtheorem{corintro}[thmintro]{Corollary}
\renewcommand*{\thecorintro}{\Alph{corintro}}

\theoremstyle{remark}

\newaliascnt{remark}{theorem}
\newtheorem{claim}[theorem]{Claim}
\newtheorem{remark}[remark]{Remark}
\newaliascnt{notation}{theorem}
\newtheorem{notation}[notation]{Notation}
\aliascntresetthe{notation}
\providecommand*{\notationautorefname}{Notation}

%\aliascntresetthe{claim}
\providecommand*{\claimautorefname}{Claim}

\aliascntresetthe{remark}
\providecommand*{\remarkautorefname}{Remark}

\newtheorem*{claim*}{Claim}

\theoremstyle{definition}
\newaliascnt{definition}{theorem}
\newtheorem{definition}[definition]{Definition}
\aliascntresetthe{definition}
\providecommand*{\definitionautorefname}{Definition}

\newaliascnt{example}{theorem}
\newtheorem{example}[example]{Example}
\aliascntresetthe{example}
\providecommand*{\exampleautorefname}{Example}

\newaliascnt{question}{theorem}
\newtheorem{question}[question]{Question}
\aliascntresetthe{question}
\providecommand*{\exampleautorefname}{Question}

\begin{abstract}
For all Dyer groups, we find an algorithm to determine when two parabolic subgroups are conjugate. Given two conjugate standard parabolic subgroup, we fully describe the conjugating elements in terms of ribbons, showing that the ribbon conjecture holds true. In particular we give a description of the normaliser of a parabolic subgroup using ribbons. We prove the standardisation property for parabolic subgroups and deduce that an arbitrary intersection of parabolic subgroups is a parabolic subgroup. 
\end{abstract}

\section{Introduction}

\subsection{Motivation}

The classes of Coxeter groups and right-angled Artin groups are generally well-understood. One common feature of Coxeter groups and right-angled Artin groups is their solution to the word problem. It was given by Tits for Coxeter groups \citep{Tits} and by Green for graph products of cyclic groups \citep{Green}. In his study of reflection subgroups of Coxeter groups \citep{Dyer1990ReflectionSubgroups}, Dyer introduces a family of groups which contains both Coxeter groups and graph products of cyclic groups. By \citep{Dyer1990ReflectionSubgroups} and \citep{PS23}, this family, which we call Dyer groups, has the same solution to the word problem as Coxeter groups and graph products of cyclic groups. It is therefore natural to study which properties of Coxeter groups and right-angled Artin groups can be extended to Dyer groups. A first answer can be found in \citep{Soergel2024Complex} where geometric actions of Dyer groups on CAT(0)
spaces are constructed that extend those of Coxeter groups on Davis–Moussong complexes \citep{moussong1988hyperbolic} and those of right-angled Artin groups on Salvetti complexes \citep{ChaDav}. In the present work we study several properties of parabolic subgroups of Dyer groups that are important conjectures in the field of Artin groups.

Similarly to Coxeter and Artin groups, Dyer groups are defined by specific presentations which can be encoded in a labelled graph.

\begin{definition}[Dyer group, Dyer system]\label{defdyer}
    Let $\Gamma$ be a finite simplicial graph with vertex set $\ver\Gamma$ and edge set $\edge\Gamma$. We suppose that $\Gamma$ comes with a labelling $f \colon \ver\Gamma \rightarrow \N_{\geq 2}\cup\{\infty\}$ of its vertices and a labelling $m\colon \edge\Gamma \rightarrow \N_{> 2}\cup\{\infty\}$ of its edges, such that for any $\{u,v\}\in \edge\Gamma$, if $f(v) \geq 3$, then $m(u,v) = \infty$. We extend the map $m\colon \edge\Gamma\rightarrow \N_{\geq 3}\cup\{\infty\}$ to $m\colon \binom{\ver\Gamma}{2}\rightarrow \N_{\geq 2}\cup\{\infty\}$ by $m(\{v,w\}) = 2$ if $\{v,w\}\notin \edge\Gamma$. The graph $\Gamma$ together with the labellings $f$ and $m$ is called a \emph{Dyer graph}.
    \begin{enumerate}
        \item Let $\Gamma$ be a Dyer graph; the associated \emph{Dyer group} is the group $D(\Gamma)$ given by
\begin{multline*}
    D(\Gamma) = \langle \ver{\Gamma} \mid v^{f(v)} = \neutre \text{ for all } v \in \ver{\Gamma} \text{ such that } f(v) < \infty, \\ \underbrace{uvu\cdots}_{m(u,v)\text{ terms}} = \underbrace{vuv\cdots}_{m(u,v)\text{ terms}} \text{ for all } \ \{u, v\} \text{ with } m(u,v)\neq\infty\rangle.
\end{multline*}
    \item A \emph{Dyer system} is a pair $(D,X)$, where $D$ is a group and $X\subseteq D$ is a generating set such that there exist a Dyer graph $\Gamma$ and a group isomorphism $D(\Gamma)\to D$ that maps $\ver\Gamma$ bijectively into~$X$. 
    \end{enumerate} 
\end{definition}

% Let $\Gamma$ be a finite simplicial graph with vertex set $\ver{\Gamma}$ and edge set $E(\Gamma)$. We suppose that $\Gamma$ comes with a labelling $f : \ver{\Gamma} \rightarrow \N_{\geq 2}\cup\{\infty\}$ of its vertices and a labelling $m: E(\Gamma) \rightarrow \N_{> 2}\cup\{\infty\}$ of its edges, such that for any $\{u,v\}\in E(\Gamma)$, if $f(v) \geq 3$ then $m(u,v) = \infty$. We extend the map $m: E\rightarrow \N_{\geq 3}\cup\{\infty\}$ to $m: \binom{V}{2}\rightarrow \N_{\geq 2}\cup\{\infty\}$ by $m(\{v,w\}) = 2$ if $\{v,w\}\notin E$. The graph $\Gamma$ is called a \emph{Dyer graph}, and the associated \emph{Dyer group} is the group $D(\Gamma)$ given by
% \begin{multline*}
%     D(\Gamma) = \langle \ver{\Gamma} \mid v^{f(v)} = \neutre \text{ for all } v \in \ver{\Gamma} \text{ such that } f(v) < \infty, \\ \underbrace{uvu\cdots}_{m(u,v)\text{ terms}} = \underbrace{vuv\cdots}_{m(u,v)\text{ terms}} \text{ for all } \ \{u, v\} \text{ with } m(u,v)\neq\infty\rangle.
% \end{multline*}

The reader unfamiliar with Coxeter groups or right-angled Artin groups may consider the following definition. A \emph{Coxeter group} is a Dyer group whose Dyer graph $\Gamma$ satisfies $f(v) = 2$ for every $v \in \ver{\Gamma}$. We will usually denote it by $W(\Gamma)$. A \emph{right-angled Artin group} (RAAG) is a Dyer group whose Dyer graph $\Gamma$ satisfies $f(v)=\infty$ for every $v\in \ver{\Gamma}$.\par

Amongst the subgroups of a Dyer group, parabolic subgroups play a special role. 

\begin{definition}[parabolic subgroups]
    \label{def:parabolic}
    Let $(D,X)$ be a Dyer system; for every $Y\subseteq X$, the \emph{standard parabolic subgroup} of $(D,X)$ associated with~$Y$ is the subgroup $D_Y$ of $D$ generated by~$Y$. \par 
    
    We call \emph{parabolic subgroup} of $(D,X)$ any $D$-conjugate of a standard parabolic subgroup of $(D,X)$. 
\end{definition}

For the sake of light notation, we will often talk about parabolic subgroups of a Dyer group rather than a Dyer system. We should however note that the definition of parabolic subgroup depends on the particular choice of the generating set~$X$. \par

\cite{Dyer1990ReflectionSubgroups} showed that, for every subset~$Y \subseteq \ver{\Gamma}$, the subgroup $D_Y$ is isomorphic to the Dyer group~$D(\Gamma_Y)$, where $\Gamma_Y$ is the full subgraph of $\Gamma$ spanned by~$Y$. %Consequently, we will write interchangeably $D_Y$ or $D(\Gamma')$ to denote the parabolic subgroup of $D(\Gamma)$ generated by $Y =\ver{\Gamma'}$.

\subsection{Statement of results} Following \citep{K94} and \citep{P97}, we give an algorithm which decides whether two standard parabolic subgroups of a Dyer system $(D,X)$ are conjugate %In fact, $D_T$ and $D_{T'}$ are conjugate if and only if the generating set $T$ and $T'$ are conjugate. 
(we refer to Theorem~\ref{equivalentconditions} for the complete statement):

\begin{thmintro}[Theorem~\ref{equivalentconditions}]
    Let $(D,X)$ be a Dyer system, let $Y,Y'\subseteq X$; the following conditions are equivalent:
    \begin{enumerate}
        \item there exists $\alpha\in D$ such that $\conjugate{D_Y}{\alpha} = D_{Y'}$;
        \item there exists $\beta \in D$ such that $\conjugate Y\beta = Y'$.
    \end{enumerate}
    Moreover, there exists an algorithm that checks whether either of the two equivalent conditions is satisfied. 
\end{thmintro}

The proof of the theorem is similar in spirit to Paris's characterisation of conjugacy between parabolic subgroups of Artin and Coxeter groups~\cite[Theorem 4.1]{P97}: starting from a Dyer system, we algorithmically construct a graph that encodes the conjugacy relation between standard parabolic subgroups.\par

Once assessed when two standard parabolic subgroups $D_Y$ and $D_{Y'}$ are conjugate, one may classify the elements that conjugate $D_Y$ into $D_{Y'}$. Let us denote by $\operatorname{Conj}(D_Y,D_{Y'})$ the set of such elements. Amongst those elements, \emph{ribbons} play a special role, as we are to explain. We will postpone the definition of ribbons to Section~\ref{sec:ribbonProperty}: for the purpose of this introduction it is enough to think about elementary ribbons as minimal elements that conjugate subsets of standard generators, and ribbons are compositions of elementary ribbons. We denote by $\operatorname{Ribb}(Y,Y')$ the subset of $\operatorname{Conj}(D_Y,D_{Y'})$ consisting of those conjugating elements that can be expressed as ribbons. Ribbons have been studied in the context of Coxeter, Artin and more generally Garside groups \citep{Godelle2010}. Thanks to results of \cite{Godelle2003} and \cite{P97}, we know that any element that conjugates a standard parabolic subgroup of an Artin group of spherical type into another can be written as the product of ribbons. This property is known as the \emph{ribbon property}. It is conjectured that all Artin groups satisfy this property, and this is known as the \emph{ribbon conjecture}. Our next result shows this property for Dyer groups.

\begin{thmintro}[Theorem~\ref{thm:ribbonConj}]
    Let $(D,X)$ be a Dyer system and let $Y,Y'\subseteq X$. Then,
    \[
    \Conj(D_Y,D_{Y'})=\Ribb(Y,Y') \cdot D_Y.
    \]
\end{thmintro}

In particular, when $Y'=Y$, we obtain an explicit description of the normaliser of the standard parabolic subgroup~$D_{Y}$. Notice that, in this case, $\operatorname{Ribb}(Y,Y)$ is a subgroup of~$D$.

\begin{corintro}[Corollary~\ref{cor:normaliserStdParabolic}]
    \label{corintro:normaliser}
    Let $(D,X)$ be a Dyer system and let $Y\subseteq X$. The normaliser of $D_Y$ in $D$ decomposes as
    \[
    \operatorname{N}_{D}(D_Y)=D_Y\rtimes \Ribb(Y,Y),
    \]
where the action of $\Ribb(Y,Y)$ on $D_Y$ is given by conjugation.
\end{corintro}

Finally, we tackle the so-called \emph{standardisation property}. Let $(D,X)$ be a Dyer system, let $Y\subseteq X$ and let $P\subseteq D$ be a parabolic subgroup. Assume further that $P\subseteq D_Y$. As $(D_Y,Y)$ is a Dyer system itself, it is natural to ask whether $P$ is also a parabolic subgroup of~$D_Y$ with respect to~$Y$. For all Dyer systems, we give a positive answer to this question:

\begin{thmintro}[Theorem~\ref{thm:parabolicInsideParabolicDyer}]
    \label{thmintro:parabInsideParab}
    Let $(D,X)$ be a Dyer system. Let $g\in D$ and $Y,Z\subseteq X$ be such that $gD_{Y}g^{-1}\subseteq D_{Z}$. Then there exist $h\in D_Z$ and $Y'\subseteq Z$ such that $gD_{Y}g^{-1}=hD_{Y'}h^{-1}$.
\end{thmintro}

For Coxeter groups, the equivalent statement of Theorem~\ref{thmintro:parabInsideParab} follows as a direct consequence of results by~\cite{Solomon}, whilst for right-angled Artin groups this is due to \cite{Duncan2007}. Remarkably, the same result for Artin groups was only proved in full generality recently in \cite[Theorem 1.1]{BlufPar}, although previous works proved the theorem when restricting to some subclasses of Artin groups.\par

Theorem~\ref{thmintro:parabInsideParab} turns out to be the key step to prove another important combinatorial property of Dyer systems: the intersection of parabolic subgroups of a Dyer group is a parabolic subgroup itself. This result was already proved in \cite[Theorem 2.10]{PS23} for finite intersections, but we generalize it to arbitrary intersections. For Coxeter groups, it was proved in \citep{Qi2007}, while for RAAGs it was proved in \citep{Duncan2007}.

\begin{thmintro}[Theorem~\ref{thm:intersectionParabSubgrp}]
    \label{thmintro:intersectionParabSubgrp}
    Let $(D,X)$ be a Dyer system and let $\{P_i\}_{i\in I}$ be a family of parabolic subgroups of $D$. Then $\bigcap_{i\in I}P_i$ is a parabolic subgroup of $D$.
\end{thmintro}

As a direct consequence, we obtain that the parabolic closure is well defined.

\begin{corintro}[Corollary~\ref{cor:parabClosure}]
    Let $(D,X)$ be a Dyer system; every element of $D$ admits a \emph{parabolic closure}, i.e. for every $g\in D$ there exists an inclusion-wise smallest parabolic subgroup of $D$ containing $g$.
\end{corintro}

\subsection{Organisation of the paper}
This paper is organised as follows. In Section~\ref{background}, we state the relevant definitions and results that we will be using in the next sections.
In Section~\ref{conjugacyparabolic}, we find an algorithm to determine when two parabolic subgroups of a Dyer group are conjugate. In Section~\ref{standardisation}, we prove the standardisation property for parabolic subgroups of Dyer groups, and deduce that an arbitrary intersection of parabolic subgroups of a Dyer group is a parabolic subgroup. Finally, in Section~\ref{sec:ribbonProperty} we prove the ribbon property for Dyer groups.

\subsection{Acknowledgements} 

This article is part of the research project PID2022-138719NA-I00, financed by the Spanish Ministry of Science and Innovation MCIN/AEI/10.13039/501100011033/FEDER, UE. The third author was partially supported by the EPSRC Standard Research Grant UKRI1\-018. 

We thank Luis Paris for pointing to \citep{Solomon} as a reference to prove Theorem~\ref{thm:parabolicInsideParabolicDyer}.

\section{Background}\label{background}

In this section we will state the relevant definitions and results that we will be using in the next sections. Note that Dyer graphs, groups and systems were already defined in the Introduction (cf. Definition \ref{defdyer}).

\begin{definition}
    Let $(\Gamma,f,m)$ be a Dyer graph and let $S(X) = \{x_v^{\alpha}\mid v\in \ver\Gamma, \alpha\in \mathbb{Z}_{f(v)}\setminus \{0\}\}$, where $\mathbb{Z}_{f(v)} = \mathbb{Z}/f(v)\mathbb{Z}$ if $f(v) <\infty$ and $\mathbb{Z}_{\infty} =\mathbb{Z}$. A \emph{syllabic word} $w$ is an element in the free monoid $S(X)^* $. It is usually written as a finite sequence. For a syllabic word $w = (s_1,s_2,\dots, s_l)\in S(X)^*$, we set $\overline{w} = s_1s_2\cdots s_l\in D(\Gamma)$ and say that $\overline{w}$ is \emph{represented} by $w$. The shortest length of a syllabic word representing an element $g\in D$ is called the \emph{syllabic length}, or simply length, of $g$ and we will denote it by $l(g)$. A syllabic word $w$ is \emph{reduced} if $l = l(\overline{w})$.
\end{definition}

\begin{definition}[$M$-transformations]
Let $w\in S(X)^*$ be a syllabic word, and assume that $w$ can be written as $w=w_1\cdot(s,t)\cdot w_2$, where $w_1,w_2\in S(X)^*$, $s,t\in S(X)$ and $st\in S(X)\cup\{1\}$. Set $w'$ to be equal to $w_1\cdot(st)\cdot w_2$ if $st\neq 1$, or $w_1\cdot w_2$ if $st=1$. Then, we say that we can go from $w$ to $w'$ through an \emph{elementary $M$-transformation of type I}.

Assume that $w$ can be written as $w=w_1\cdot[s,t]_m\cdot w_2$, where $w_1,w_2\in S(X)^*$, $s,t\in S(X)$, $m\geq 2$, $\overline{[s,t]_m}=\overline{[t,s]_m}$ and $l(\overline{[s,t]_m})=m$. Set $w'=w_1\cdot[t,s]_m\cdot w_2$. Then we say that we can go from $w$ to $w'$ through an \emph{elementary $M$-transformation of type~II}.

We say that $w$ is \emph{$M$-reduced} if its length cannot be shortened by any finite sequence of elementary $M$-transformations.
\end{definition}

The following result was proved in \cite[Theorem 2.2]{PS23} and is one of the main properties of Dyer groups that we will be using.

\begin{theorem}[{{{\citealp[Theorem 2.2]{PS23}}}}]\label{Theorem_transformations}
Let $(D,X)$ be a Dyer system, then:
\begin{enumerate}
    \item for all $w\in S(X)^*$, $w$ is reduced if and only if $w$ is $M$-reduced, and
    \item for all $w,w'\in S(X)^*$, if $w$ and $w'$ are both reduced and $\overline{w}=\overline{w'}$, then we can go from $w$ to $w'$ through a finite sequence of elementary $M$-transformations of type II.
\end{enumerate}
We express this by saying that $(D,X)$ has \emph{Property $\mathcal{D}$}.
\end{theorem}
    
\begin{definition}    
    The \emph{support} of a syllabic word $w = (x_1^{\alpha_1}, x_2^{\alpha_2}, \dots, x_l^{\alpha_l})$ is $\Supp(w) = \{x_1, x_2, \dots, x_l\}$. For $g\in D$, choose a reduced syllabic word $w = (x_1^{\alpha_1}, x_2^{\alpha_2}, \dots, x_l^{\alpha_l})$ representing $g$. We define $\Supp(g)= \Supp(w) = \{x_1, x_2, \dots, x_l\}$. We also define the \emph{exponent sum} of a generator $x$ in $g$ as the sum $\sum_{i\in I}\alpha_{i}$ where
    \[
    I=\{i\mid x_i\in\Supp(g), x=x_i\}.
    \]
\end{definition}

\begin{remark}
    By \cite[Theorem 2.2]{PS23}, the definitions of $\Supp(g)$ and the exponent sum do not depend on the choice of reduced word $w$ representing $g$.
\end{remark}

% For a Dyer system $(D,X)$ with Dyer graph $(\Gamma,f,m)$ and for each $Y\subseteq X$, we write~$\Gamma_Y$ for the full subgraph spanned by $\{v\in V\mid x_v\in Y\}$, and $D_Y$ for the subgroup of $D$ generated by $Y$.

% \begin{definition}
%     The subgroups $D_Y$ with $Y\subseteq X$ are called \emph{standard parabolic subgroups} of $D$, and subgroups of the form $gD_Yg^{-1}$ for $Y\subseteq X$, $g\in D$ are called \emph{parabolic subgroups} of $D.$
% \end{definition}

%Let $(\Gamma,f,m)$ be a Dyer graph. It follows from \cite[Proposition~2.7]{PS23} that for any parabolic subgroup $gD_Yg^{-1}$, the pair $(gD_Yg^{-1},gYg^{-1})$ is a Dyer system with Dyer graph $(\Gamma_Y,f_Y,m_Y)$ where~$f_Y$ and~$m_Y$ are the restrictions of $f,m$ to the vertices and edges of $\Gamma_Y$, respectively.

It was shown in \cite[Lemma~2.5]{PS23} that standard parabolic subgroups of Dyer groups are \emph{convex}, which means that an element $g\in D$ belongs to~$D_{Y}$ if and only if its support belongs to~$Y$.

\begin{proposition}[{{{\citealp[Proposition~2.8]{PS23}}}}]\label{Proposition 2.8}
Let $(D,X)$ be a Dyer system, let $Y\subseteq X$, and $g\in D$.
\begin{enumerate}
    \item There exists a unique element $g_0$ in $gD_Y$ of minimal syllabic length, and this elements satisfies $l(g_0h)=l(g_0)+l(h)$ for all $h\in D_Y$.
    \item There exists a unique element $g_0$ in $D_Yg$ of minimal syllabic length, and this elements satisfies $l(hg_0)=l(g_0)+l(h)$ for all $h\in D_Y$.
\end{enumerate}
\end{proposition}

We note that this result is also true for the usual word length \cite[Proposition~3.5]{ParisVarghese}.

\medskip

%Dyer groups generalize other well-known families of groups such as Coxeter groups, graph products of cyclic groups and right-angled Artin groups (RAAGs). In particular, a Coxeter group is a Dyer group in which $f(v)=2$ for all $v\in V$, while a RAAG is a Dyer group in which $f(v)=\infty$ for all $v\in V$.

A Coxeter group $W$ with defining Dyer graph $\Gamma$ is called \emph{irreducible} if $\Gamma$ is connected. Otherwise, we can decompose $W$ as the direct product of the irreducible parabolic subgroups that correspond to the connected components. Finite irreducible Coxeter groups are classified via their defining graphs, as shown in Figure~\ref{coxeter}. For every finite Coxeter system $(W,X)$ there is a unique element of maximal length which we will denote by $w_0$ and which satisfies $w_0^2=1$ and $w_0 X w_0= X$. This gives a permutation $X\to X$ of order~2 which we denote by~$\delta_0$.

\begin{figure}
  \centering
  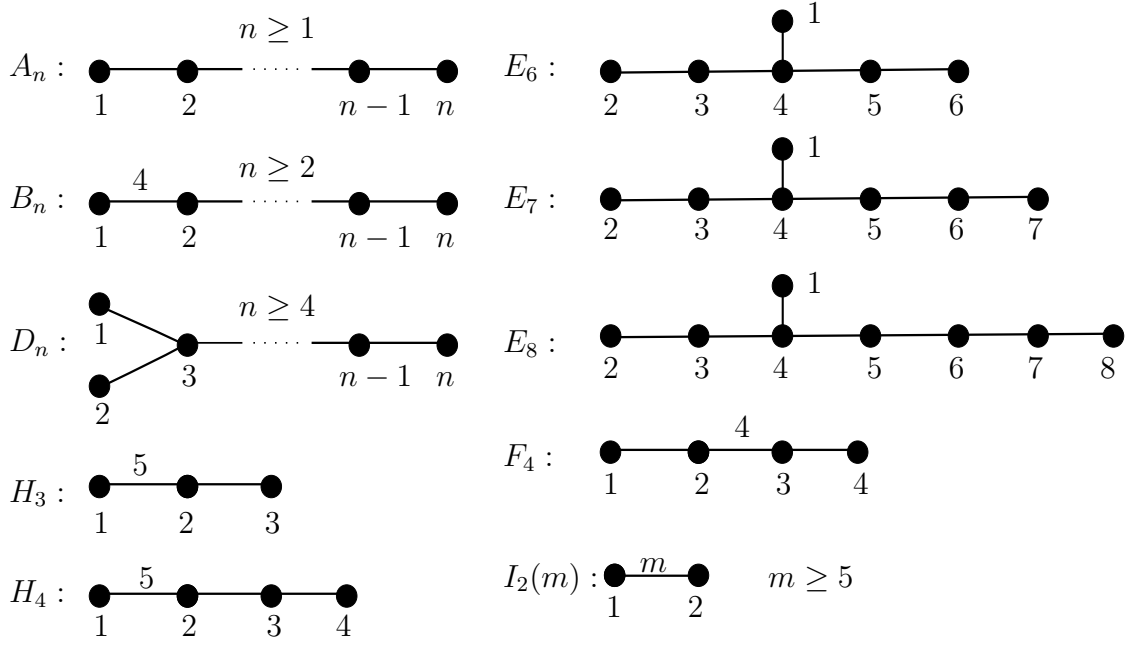
  \medskip
  \caption{Classification of irreducible Coxeter graphs of finite type. The unlabelled edges are labelled with $3$. The labelling in the vertices is only meant for counting and is not to be confused with vertex labelling of Dyer graphs.}
  \label{coxeter}
\end{figure}

Suppose that $(D,X)$ is a Dyer system and $\Gamma$ is its corresponding Dyer graph. Given $Y\subseteq X$ we will denote by $Y_2$ the set of elements of order $2$ in $Y$, and $\Gamma_{Y_2}$ be the induced subgraph of $\Gamma$. Similarly, we will denote by $Y_\infty$ the set of elements of infinite order, and by $Y_p$ the set of elements that have neither order $2$ or infinite. Notice that $Y=Y_2\sqcup Y_\infty\sqcup Y_p$. Also notice that $D_{Y_2}$ is a Coxeter group while $D_{Y_\infty}$ is a right-angled Artin group. We will usually denote $D_2:=D_{Y_2}$ and so on.

For an element $g$ of a group, we will denote its order by $o(g)$.

A Dyer group $D$ with Dyer graph $\Gamma$ is said to be of \emph{spherical type} if there are no edges labelled by $\infty$ and $D_2$ is finite. Then we have that that $D=D_{2}\times D_p \times D_\infty$, $D_\infty=\mathbb{Z}^{|Y_\infty|}$ and $D_p$ is a direct product of finite cyclic groups.

\section{Conjugacy between parabolic subgroups}\label{conjugacyparabolic}

The aim of this section is to answer the following question.

\begin{question}
Let $(D,X)$ be a Dyer system, and let $Y,Y'\subseteq X$ be two subsets of standard generators. When are the standard parabolic subgroups $D_Y,D_{Y'}$ conjugate?
\end{question}

Suppose that $(D,X)$ is a Dyer system with defining Dyer graph $\Gamma$. First, we will construct a graph $G$ that will encode the conjugacies between parabolic subgroups. The vertices of $G$ are the subsets $Y\subseteq X$. An edge of $G$ is a triple $(Y,t,t')$ satisfying the following:
\begin{enumerate}[label=(\alph*)]
\item $Y\subseteq X$
\item Both $t$ and $t'$ have order 2 and belong to the same connected component~$\Gamma_0$ of $\Gamma_{Y}$ and to the same connected component $\Gamma_0'$ of $\Gamma_{Y_2}$, and moreover these two coincide, $\Gamma_0=\Gamma_0'$.
\item $\Gamma_0\in\{A_l\mid l\geq2\}\cup\{D_l\mid l\geq5,l\text{ odd}\}\cup\{E_6\}\cup\{I_2(p)\mid p\geq5,p\text{ odd}\}$. These graphs correspond to what we call \emph{twistable} finite irreducible Coxeter groups. 
\item $t'=\delta_0(t)$ i.e. $t'=w_0 t w_0$ and $t\neq t'$, where $w_0$ is the element of maximal length in $W_{\Gamma_0}=D_{\Gamma_0}$.
\end{enumerate}

The edge $(Y,t,t')$ joins $Y\setminus \{t\}$ with $Y\setminus \{t'\}$. Notice that in this case one has $Y\setminus \{t\}=w_0(Y\setminus\{t'\})w_0$, since:
\begin{itemize}
    \item the condition on the connected components in (2) means precisely that every generator of order bigger than 2 in $Y$ commutes with every element in $D_{\Gamma_0}$, and
    \item $(Y\setminus \{t\})_2=w_0(Y\setminus\{t'\})_2 w_0$.
\end{itemize}

\medskip

\begin{remark}\label{G2inG}
This construction is analogous to the one described for Coxeter systems in \cite[Section 4]{P97}, originally given in Krammer's algorithm \cite[Chapter 3.1]{K94}, to decide whether two parabolic subgroups of a Coxeter or Artin system are conjugate. Indeed, it is straightforward to check that the full subgraph $G_2$ of $G$ spanned by the subsets $Y\subseteq X_2$ coincides with the graph associated to the Coxeter system $(D_{X_2},X_2)$ by Krammer's construction. In fact, $G_2$ is a union of connected components of $G$. Also notice that a vertex corresponding to a subset $Y\subseteq X$ can be connected to a different vertex in $G$ only if $\Gamma_Y$ contains a connected component consisting only of order~2 elements.
\end{remark}

\begin{example}
Consider first the Coxeter group with graph
\begin{center}
\begin{tikzpicture}[scale=0.8]
\draw(0,0)--(3,0);
\filldraw(0,0) circle[radius=0.1];
\filldraw(1,0) circle[radius=0.1];
\filldraw(2,0) circle[radius=0.1];
\filldraw(3,0) circle[radius=0.1];

\node at (-0.5,0)[left] {$A_4$=};
\end{tikzpicture}
\end{center}
where as before the unlabelled edges are labelled with 3. This is a finite irreducible twistable Coxeter group. Figure \ref{graphb5} is a picture of its associated graph~$G$, where we represent each subset of vertices as the full subgraph they span (in black). Subgraphs in the same row have the same number of vertices, and edges can only join subgraphs in the same row. For better visualization, the subgraphs are placed as follows. Using the same notation as above, for an edge $(Y,t,t')$, the subgraph placed below the middle of the edge is that spanned by~$Y$. The edge joins two different subgraphs $Y\setminus\{t\}$ and  $Y\setminus\{t'\}$ that are permuted by the twist in $Y$. Note that there can be several edges labelled with the same subset $Y$.

\begin{figure}[h]
\centering
\includegraphics[width=\linewidth]{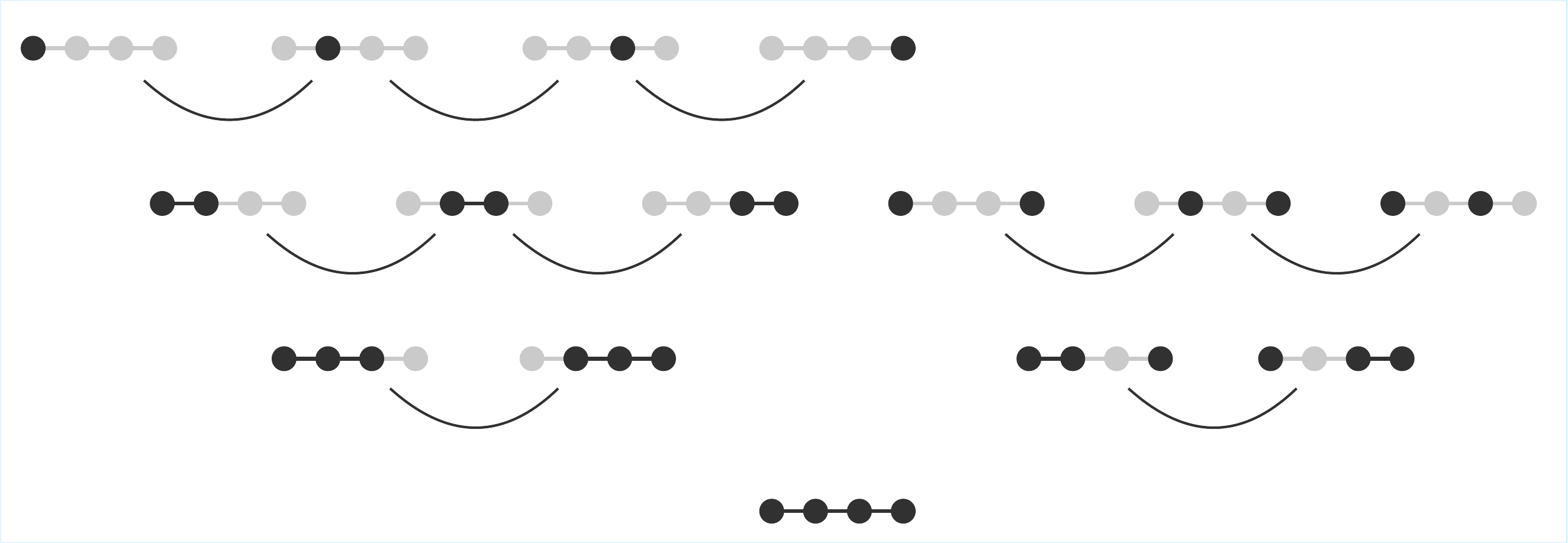}
  \caption{Graph $G$ corresponding to the Coxeter group of type $A_4$.}
  \label{graphb5}
\end{figure}

Now let us add a new generator with order bigger than 2, which we will color with white to distinguish it from the others:
\begin{center}
\begin{tikzpicture}[scale=0.8]
\draw(0,0)--(3,0);
\filldraw(0,0) circle[radius=0.1];
\filldraw(1,0) circle[radius=0.1];
\filldraw(2,0) circle[radius=0.1];
\filldraw(3,0) circle[radius=0.1];
\draw(3,0)--(3,1);
\filldraw[fill=white](3,1) circle[radius=0.1];
\node at (3,0.5)[right] {$\infty$};
\node at (-0.5,0)[left] {$\Gamma$=};
\end{tikzpicture}
\end{center}
The graph $G$ associated to the corresponding Dyer system is the disjoint union of the graph in Figure \ref{graphb5} and the graph in Figure \ref{graphb5dot}. Note that the effect of adding the new generator is that some edges disappear. As we will prove, two different connected subgraphs containing a generator of order bigger than 2 cannot be conjugate.

\begin{figure}[h]
\centering
\includegraphics[width=\linewidth]{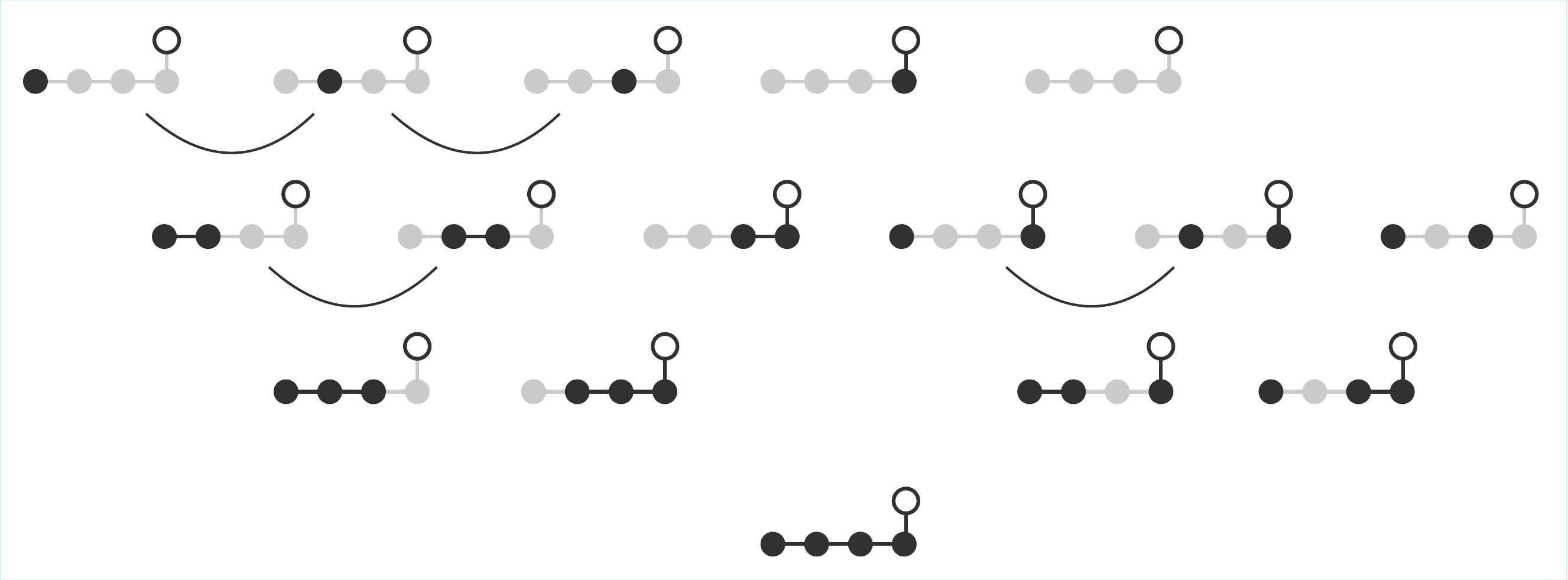}
  \caption{Part of the graph $G$ corresponding to the Dyer system with graph $\Gamma$.}
  \label{graphb5dot}
\end{figure}
    
\end{example}

We will now prove some lemmas that we will use in the proof of the main theorem of this section.

\begin{lemma}\label{removegenerator}
Let $G$ be a group that satisfies Property $\mathcal{D}$, and let $w,w'$ be two reduced equivalent words in $G$. Let $s$ be a standard generator that can be involved only in commutation relations. Then $w(\hat{s})$ and $w'(\hat{s})$ are also equivalent, where $w(\hat{s})$ denotes the word $w$ with all the occurrences of $s$ removed.
\end{lemma}

\begin{proof}
We prove it by induction on the number $r$ of $M$-transformations that we need to get from $w$ to $w'$.

Suppose that $r=1$. We can write $w = w_1\cdot [a,b]_{m(a,b)}\cdot w_2$ and $w' = w_1\cdot [b,a]_{m(a,b)}\cdot w_2$. Since $\overline{[a,b]_{m(a,b)}(\hat{s})} = \overline{[b,a]_{m(a,b)}(\hat{s})}$, it follows that $\overline{w(\hat{s})} = \overline{w'(\hat{s})}$. In particular, by Property~$\mathcal{D}$, one can use $M$-transformations of type II to get from $w(\hat{s})$ to $w'(\hat{s})$. More specifically, if $s\in\{a,b\}$, we have $m(a,b)=2$ and so $w(\hat{s}) = w'(\hat{s})$ (as words) and if $s\notin\{a,b\}$, one needs one $M$-transformation of type~II to get from~$w(\hat{s})$ to~$w'(\hat{s})$.

Now assume that the result holds for $r-1$. Let $w''$ be the word we get after applying $r-1$ $M$-transformations to~$w$, which is one $M$-transformation apart from~$w'$. Note that~$w''$ is also reduced. By the induction hypothesis, we see that~$w(\hat{s})$ and~$w''(\hat{s})$ are equivalent. And then by the case $r=1$, we have that $w''(\hat{s})$ and $w'(\hat{s})$ are equivalent as well. By transitivity, we have the desired result. 
\end{proof}

\begin{remark}\label{remark_hat}
Let $(D,X)$ be a Dyer system. Let $g\in D$, and let $w$ be a syllabic reduced word representing $g$, so that $g=\overline{w}$. Let $\hat{w}$ denote the same word without every occurrence of a standard generator of order greater than~2. Then by Lemma~\ref{removegenerator}, the element $\overline{\hat{w}}\in D$ does not depend on the choice of $w$, hence we may denote it $\hat{g}:=\overline{\hat{w}}$. We shall keep this notation from now on. We also note that the map $\hat{\phantom{o}}\colon\to D_2$ given by $g\mapsto \hat{g}$ is a group homomorphism.
\end{remark}

\begin{theorem}\label{equivalentconditions}
Let $(D,X)$ be a Dyer system.  Let $Y, Y' \subseteq X$. Then the following statements are equivalent:

\begin{enumerate}
    \item\label{cond subsets} There exists $\alpha\in D$ such that $\alpha Y \alpha^{-1} = Y'$.
    \item\label{cond subgroups} There exists $\beta \in D$ such that $\beta D_Y \beta^{-1} = D_{Y'}$.
    \item\label{cond graph} $Y$ and $Y'$ are in the same connected component of $G$.
\end{enumerate}
\end{theorem}

\begin{proof}

($\ref{cond subsets} \Leftrightarrow \ref{cond subgroups}$). $\ref{cond subsets} \Rightarrow \ref{cond subgroups}$ is trivial. Let us see the other implication. We suppose $\beta D_Y \beta^{-1} = D_{Y'}$. We can find $v$ of minimal length such that $\beta D_Y = vD_Y$ (Proposition~\ref{Proposition 2.8}). We can write $\beta=vu=u'v$, $u\in D_Y$ and $u'\in D_{Y'}$, because 
$$\beta D_Y \beta^{-1}= v D_Y v^{-1}=D_{Y'}.$$
This implies, also by Proposition~\ref{Proposition 2.8}, that for all $s\in S(Y)$, $l(v) + 1 = l(vs) = l(vsv^{-1}v) = l(vsv^{-1}) + l(v)$. Hence $l(vsv^{-1}) = 1$ and so $vsv^{-1}\in S(X)$.
So $vsv^{-1}\in D_{Y'}\cap S(X) = S(Y')$, because $D_{Y'}$ is convex \cite[Lemma 2.5]{PS23}. 

It follows that $vS(Y)v^{-1} \subseteq S(Y')$. 
Suppose that $v y v^{-1} = y'^{r_1}$ for $y\in Y$ and $y'\in Y'$, so $v^{-1} y'^{r_1}v = y$. We also have that $v^{-1} y' v = \hat{y}^{r_2}$ for some $\hat{y}\in Y$, so $1=r_1r_2$ and $y=\hat{y}$.
Then we do a case-by-case analysis. If $o(y)=2$, then $y^{-1}=y$ and $r_1 = r_2 =1$.
If $o(y)>2$, then we have that $o(y')> 2$ and the exponent sum of $y$ in $v^{-1}y'^{r_1}v$ is either $0$ (if $y\neq y'$) or $r_1$ (if $y= y'$).
%as $y$ and $y'$ only participate in even order relations. 
But since $v^{-1}y'^{r_1}v = y$, this implies $r_1 =1$ and $y=y'$. This shows $v Y v^{-1} \subseteq Y'$. Since $v D_Y v^{-1}=D_{Y'}$ and no subset of $Y'$ is a generating set for $D_{Y'}$, the fact that the conjugation is an isomorphism implies that  $v Y v^{-1} = Y'$.
Note that we have shown that two different generators of order bigger than 2 cannot be conjugate.

\medskip
\noindent
($\ref{cond subsets} \Leftrightarrow \ref{cond graph}$).
$\ref{cond subsets} \Leftarrow \ref{cond graph}$ is trivial. We have to prove that if $\alpha Y\alpha^{-1} = Y'$ then $Y$ and $Y'$ are in the same connected component of $G$. If $Y=Y'$, this is trivial, so we will assume that $Y\neq Y'$. We start noticing that the generators of order different than 2 in a Dyer group commute with every other generator or do not share a relation, so $(Y)_{p\cup \infty}= (Y')_{p\cup \infty}$.

\medskip

\begin{claim} \label{claim:centraliser}
Let $x\in X$ with $o(x)>2$. Then the centraliser of $x$ in $D$ is $\operatorname{C}_D(x)= D_{\{x\}\cup\{x\}^\perp}= \langle x \rangle \times D_{\{x\}^\perp}$, where $\{x\}^\perp$ is the set of generators that commute with $x$. Moreover, the set of generators which are conjugates of $x$ is $\{x\}$.
\end{claim} 

\medskip

\begin{proof}[Proof of Claim~\ref{claim:centraliser}] The last statement was shown above. Assume $\alpha x \alpha^{-1}=x$. Without loss of generality, we may assume that $\alpha$ is of minimal length in the double coset $D_{\{x\}}\alpha D_{\{x\}}$. Choose a reduced word $w$ representing~$\alpha$, then this means that the words $w\cdot x$ and $x\cdot w$ are reduced. Then by Theorem~\ref{Theorem_transformations} we can go from one to another by a sequence of $M$-transformations.  Now if $o(x)>2$ then the only relations in which~$x$ is involved are commutators, thus~$x$ must commute with the generators in the support of $\overline{w}=\alpha$, proving the claim.
\end{proof}

\begin{claim} \label{claim:fixedcomponents}
    If~$Y_1$ defines a connected component in the Dyer graph of $D$ and contains $x\in Y_1$ with $o(x)>2$, then there cannot exist $\alpha\in D$, $Y_2\subseteq X$ such that $\alpha Y_1\alpha^{-1}=Y_2$ with $Y_1\neq Y_2$.
\end{claim}

\begin{proof}[Proof of Claim~\ref{claim:fixedcomponents}]
Assume the contrary. Let~$\Gamma_1$ and~$\Gamma_2$ be the Dyer graphs of~$Y_1$ and~$Y_2$ respectively, which are connected. There is $\alpha\in D$ such that $\alpha t_1\alpha^{-1}=t_2$ for $t_1\in Y_1$ and $t_2\in Y_2\setminus Y_1$. Let $m_1,m_2\in \{2,\infty\}$ be the length of the relations involving~$x$ and~$t_1$, and~$x$ with~$t_2$, respectively. We also know by Claim \ref{claim:centraliser} that $\alpha\in \operatorname{C}_D(x)$ and in particular, if  we conjugate $t_1 x=x t_1$ by $\alpha$ we obtain $t_2 x=x t_2$, so we have either $m_1=m_2=\infty$ or $m_1=m_2= 2 $. In the first case, $t_1$ and $t_2$ cannot belong to the support of $\alpha\in \operatorname{C}_D(x) = \langle x \rangle \times D_{\{x\}^\perp}$. By Theorem~\ref{Theorem_transformations}, we should be able to pass from any reduced word representing $\alpha t_1 \alpha^{-1}$ to $t_2$ using only $M$-transformations of type II.   Since the relations of our group involving more than one generator are homogeneous, $t_1,t_2\not\in \mathrm{Supp}(\alpha)$ implies that $\alpha$ cannot conjugate $t_1$ to $t_2$, having a contradiction. Therefore, we are in the second case. Since $Y_1$ and $Y_2$ define each a connected component, there must exist a path $(x=y_0,y_1,\dots, y_{k-1}, y_k=t_1)$ between $x$ and $t_1$ inside $\Gamma_1$ which is sent by conjugation by $\alpha$ to a path $(x=y'_0,y'_1,\dots,y'_{k-1},y'_k=t_2)$ between $x$ and $t_2$ inside $\Gamma_2$. We can assume that these paths are geodesic. This means precisely that for all $i\in\{0,\dots,k\}$, $y_i$ (resp. $y'_i$) does not commute with any $y_j$ (resp. $y'_j$) whenever $|i-j|=1$, and it does commute with every $y_j$ (resp. $y'_j$) whenever $|i-j|>1$.

%Let $n_1=n_2\neq 2$ be the lengths of the relation between $y_k$ and $t_1$, and $y'_{k}$ and $t_2$, respectively. Since they are conjugates, $o(y_k)=o(y'_{k})$. If $o(y_k)=o(y'_{k})>2$, then $n_1=n_2=\infty$ and by Claim 1 we have $y_k=y'_k$ and $\alpha\in C(y_k)$, being impossible to conjugate $t_1$ to $t_2$ as before. Thus, $o(y_k)=o(y'_{k})=2$.

%We may assume that all $y_i$ and $y'_i$ have order 2. Indeed, consider the largest $j$ such that $o(y_{j})>2$ (this set is non-empty as $o(x)>2)$, then by Claim~1 $y_{j}=y'_{j}$ and we can rename it as $x$ and consider only the $y_i$ with $i>j$.

%Consider now the smallest $l$ such that $y_l\neq y'_l$. We show by induction that for all $i=1,\dots,l-1$, $y_i$ does not belong to the support of $\alpha$. Indeed, since $\alpha\in C(x)$ and $y_1$ does not commute with $y_1$, then by Claim~1, $y_1\notin \Supp(\alpha)$. Now assume that $y_i\notin \Supp(\alpha)$. Since $\alpha$ fixes $y_i$, then by \cite[Theorem 2.2]{PS23}, $y_{i+1}$ cannot be in the support either, since we should be able to use only $M$-transformations.

%Therefore, $y_{l-1}$ is not in the support of $\alpha$ and it is fixed by $\alpha$, which implies that $y_l$ and $y'_l$ cannot be in the support on $\alpha$ and therefore they cannot be conjugate. This finishes the proof of the claim. 

Since both paths start at the same point $x$ and end at different points $t_1\neq t_2$, there is $i\in\{0,\dots,k-1\}$ which is the smallest index such that $y_i=y'_i$ and $y_{i+1}\neq y'_{i+1}$. Then, $\alpha\in \operatorname{C}_D(y_i)$ and we distinguish two cases.

If $o(y_i)>2$, then by Claim \ref{claim:centraliser} neither~$y_{i+1}$ or~$y'_{i+1}$ are in the support of~$\alpha$, because they do not commute with~$y_i$. But we have seen before that in this case $\alpha$ cannot conjugate~$y_{i+1}$ to~$y'_{i+1}$, which yields a contradiction.

If $o(y_i)=2$, then $i\geq1$ and we are going to show that for each $1\leq j\leq i$, $y_j\notin\Supp(\alpha)$ and all the letters in $\Supp(\alpha)$ must commute with $y_j$. For $i=1$, notice that $y_1$ does not commute with $x$, so $y_1\notin\Supp(\alpha)$ because we have seen that  $\alpha\in \operatorname{C}_D(x)=\langle x\rangle\times D_{\{x\}^\perp}$. But we also have $\alpha y_1=y_1\alpha$, so choosing a reduced word for a representative of minimal length in $D_{Y_1} \alpha D_{Y_1}$, we have that, like in the proof of Claim \ref{claim:centraliser}, by Theorem~\ref{Theorem_transformations}  all letters in $\Supp(\alpha)$ must commute with $y_1$. But then, since $y_2$ does not commute with $y_1$, we have $y_2\notin\Supp(\alpha)$, and since $\alpha$ also fixes $y_2$, with the same reduced word we deduce that all letters in $\Supp(\alpha)$ commute with $y_2$. We can again use the same argument until we reach $y_i$.

Now we have $\alpha y_{i+1}\alpha^{-1}= y'_{i+1}$, and neither $y_{i+1}$ nor $y'_{i+1}$ commute with $y_i$, therefore they do not belong to the support of~$\alpha$ and as before it is impossible that~$\alpha$ conjugates one to another, yielding a contradiction. This finishes the proof of the claim.
\end{proof}

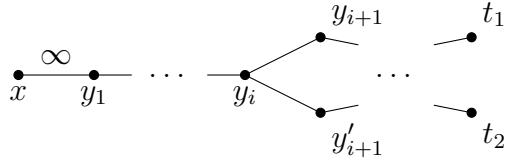
\begin{figure}
    \centering
    \begin{tikzpicture}
    \draw(0,0)--(1.5,0);
    \draw(2.5,0)--(3,0)--(4,0.5)--(4.5,0.4);
    \draw(3,0)--(4,-0.5)--(4.5,-0.4);
    \draw(5.5,0.4)--(6,0.5);
    \draw(5.5,-0.4)--(6,-0.5);
\filldraw(0,0) circle[radius=0.06];
\filldraw(1,0) circle[radius=0.06];
\filldraw(3,0) circle[radius=0.06];
\filldraw(4,0.5) circle[radius=0.06];
\filldraw(4,-0.5) circle[radius=0.06];
\filldraw(6,0.5) circle[radius=0.06];
\filldraw(6,-0.5) circle[radius=0.06];
\node at (0,0)[below]{$x$};
\node at (1,0)[below]{$y_1$};
\node at (3,0)[below]{$y_i$};
\node at (2,0){$\dots$};
\node at (5,0){$\dots$};
\node at (4,0.5)[above right]{$y_{i+1}$};
\node at (4,-0.5)[below right]{$y'_{i+1}$};
\node at (6,0.5)[above right]{$t_1$};
\node at (6,-0.5)[below right]{$t_2$};
\node at (0.5,0)[above]{$\infty$};
    \end{tikzpicture}
    \caption{The paths from $x$ to $t_1$ and $t_2$. The first bifurcation point occurs at $y_i$.}
    \label{claim2fig}
\end{figure}

\medskip

Next, we prove that when we conjugate a set of generators to another, we can always do the same conjugation using an element of $D_2$. In particular, if we conjugate two different sets with all elements of order 2, then we can do the conjugation inside the corresponding Coxeter group. This will allow us to finish the proof using Krammer's result as explained afterwards.

\begin{claim} \label{claim:conjbycoxeter}
Let $Y,Y'\subseteq X$ and $\alpha\in D$ such that $\alpha Y\alpha^{-1} =Y'$, then there is $\alpha'\in D_2$ such that $\alpha' y \alpha'^{-1} = \alpha y \alpha^{-1}$ for every $y\in Y$.
\end{claim} 

\begin{proof}[Proof of Claim~\ref{claim:conjbycoxeter}]
For every $x\in Y$ we are going to pick some reduced $w_x$ representing~$\alpha$. We will prove that $\alpha x \alpha^{-1}= \overline{\hat{w}_x} x \left(\overline{\hat{w}_x}\right)^{-1}$ for every $x\in Y$. Since $\overline{\hat{w}_x}=\hat{\alpha}$ does not depend on $x$ (Remark~\ref{remark_hat}), then we can set $\alpha':=\hat{\alpha}$.

If $o(x)>2$, by Claim \ref{claim:centraliser} we know that $\alpha\in \operatorname{C}_D(x)=\langle x\rangle\times D_{\{x\}^\perp}$, so we can choose~$w_x$ to be a word where every letter commutes with $x$ or is a power of~$x$. Then, $\overline{\hat{w}_x}$ still gives an element in~$\operatorname{C}_D(x)$.

Now, for every $x\in Y$ of order 2, we have $\alpha x \alpha^{-1} =x'$ for some  $x'\in Y'$. Let $w$ be a word representing an element in $D_Y\alpha D_{Y'}$ which is of minimal length, so that $x \cdot w = w \cdot x'$ and both words of the equality are reduced. By Lemma~\ref{removegenerator} we can remove from~$w$ any power of a generator of order bigger than 2 to get a new word~$w_x$ and still have $x \overline{\hat{w}_x} =\overline{x\hat{w}_x}=\overline{\hat{w}_xx'}=\overline{\hat{w}_x} x'$, as we wanted. This finishes the claim.  
\end{proof}

Now we have all what we need to finish the proof. Since $\alpha Y \alpha^{-1}=Y'$, we have on the one hand that $(Y)_{p\cup\infty}=(Y')_{p\cup\infty}$. On the other hand, we have $\alpha Y_2\alpha^{-1}=Y'_2$, thus by Claim \ref{claim:conjbycoxeter} we can find $\alpha'\in D_2$ with $\alpha' Y_2\alpha'^{-1}=Y'_2$. By \cite[Corollary 3.1.7]{K94}, this means that~$Y_2$ and~$Y'_2$ are in the same connected component of~$G_2$, thus they are also in the same connected component of $G$ by Remark~\ref{G2inG}. Without loss of generality, let us assume that there is an edge $(Y''_2,t,t')$ in $G_2$ joining them. Recall that $t\neq t'$. Let us see that there is an edge in~$G$ between~$Y$ and~$Y'$. We verify the four conditions that we need to have an edge $(Y'',t,t')$ in $G$ between~$Y$ and~$Y'$.
\begin{enumerate}[label=(\alph*)]
\item $Y'':=Y''_2\cup (Y)_{p\cup\infty}\subseteq X$.
\item We want to prove that the connected component $\Gamma_0$ of $\Gamma_{Y''_2}$ containing~$t$ and~$t'$ is also a connected component of $\Gamma_{Y''}$. Suppose that it is not case. Then, we will have $x\in Y''$ of order bigger than 2 in the connected component of $\Gamma_{Y''}$ containing~$t$ and~$t'$. By Claim \ref{claim:fixedcomponents}, this means that $Y_2\sqcup\{x\}$ cannot be conjugate to a different subset of $X$, which contradicts that $t\neq t'$.
\end{enumerate}

And the (c) and (d) items are trivially satisfied too by \cite[Section 3]{K94}, hence $(Y'',t,t')$ is an edge in $G$ joining $Y$ and $Y'$. 
\end{proof}

\begin{corollary}
Let $(D,X)$ be a Dyer system with Dyer graph $(\Gamma,f,m)$, let \mbox{$u,v\in\ver\Gamma$} such that $u\neq v$. The following statements are equivalent:
\begin{enumerate}
    \item There exists $\alpha\in D$ such that $\alpha x_u \alpha^{-1}=x_v$.
    \item There is path in $\Gamma$ with sequence of vertices \mbox{$(u=v_1,v_2,\dots,v_{k-1},v_k=v)$} and sequence of edges $(e_1,\dots,e_{k-1})$ such that $f(v_i)=2$ for every \mbox{$i=1,\dots,k$} and $m(e_i)$ is finite and odd for every $i=1,\dots,k-1$.
\end{enumerate}
\end{corollary}

\section{Standardisation of parabolic subgroups}\label{standardisation}

In this section we study the following question.

\begin{question}
Let $(D,X)$ be a Dyer system. Let $g\in D$ and $Y,Z\subseteq X$ such that $gD_{Y}g^{-1}\subseteq D_{Z}$. Do there exist $h\in D_Z$ and $Y'\subseteq Z$ such that $gD_{Y}g^{-1}=hD_{Y'}h^{-1}$?
\end{question}

We call this property \emph{standardisation property} and it is already known to be true both for Artin groups \citep{BlufPar} and for Coxeter groups. Although it has been of general knowledge for a long time, we were not able to find a proof in the literature for Coxeter groups. We note that our proof for Dyer groups trivially includes the case of Coxeter groups, although a result solely for Coxeter groups can be proved with the same structure using \cite[Lemma~2]{Solomon} in place of \cite[Lemma~6.1]{PS23}.

\begin{theorem}
    \label{thm:parabolicInsideParabolicDyer}
Let $(D,X)$ be a Dyer system. Let $g\in D$ and $Y,Z\subseteq X$ such that $gD_{Y}g^{-1}\subseteq D_{Z}$. Then there exist $h\in D_Z$ and $Y'\subseteq Z$ such that $gD_{Y}g^{-1}=hD_{Y'}h^{-1}$.
\end{theorem}

\begin{proof}
We write $g\in G$ in the form $g=atb$ where $a\in D_Z, b\in D_{Y}$ and $t$ is of minimal syllabic length in the double coset $D_ZtD_Y$ (we note that for the case of Coxeter groups, this decomposition is unique \cite[Chap. 4, Exercise 1.3]{Bourbaki}). We have
\[
gD_{Y}g^{-1}=atbD_{Y} b^{-1}t^{-1}a^{-1}=atD_{Y} t^{-1}a^{-1} \subseteq D_Z
\]
and $a\in D_Z$, thus $tD_{Y} t^{-1}\subseteq D_Z$. Now, by \cite[Lemma 6.1]{PS23} we have
\[
tD_{Y} t^{-1}=tD_{Y} t^{-1}\cap D_Z = D_{Y'}
\]
where $Y'=t Y t^{-1}\cap Z$. Therefore we have $gD_{Y}g^{-1}=aD_{Y'}a^{-1}$ with $Y'\subseteq Z$ and $a\in D_Z$, as we wanted.
\end{proof}

We shall now use Theorem~\ref{thm:parabolicInsideParabolicDyer} to see that the intersection of an arbitrary family of parabolic subgroups is again a parabolic subgroup. Note that the result is already known when the Dyer system is of finite type~\cite[Theorem 2.10]{PS23}, although it is true in general that the intersection of \emph{finitely many} parabolic subgroups is a parabolic subgroup, with no restrictions on the type of the Dyer system~\cite[Lemma 6.2]{PS23}.

\begin{theorem}
    \label{thm:intersectionParabSubgrp}
    Let $(D,X)$ be a Dyer system and let $\{P_i\}_{i\in I}$ be a family of parabolic subgroups of $D$. Then $\bigcap_{i\in I}P_i$ is a parabolic subgroup of $D$.
\end{theorem}

\begin{proof}
    Our proof follows the lines of the proof the result for Dyer systems of finite type~\cite[Theorem 2.10]{PS23}. If $I$ is empty, then $\bigcap_{i\in I}P_i=D$, which is a parabolic subgroup itself. Hence we may assume that $I$ in not empty. Let $\mathcal F$ be the set of all possible intersections of finitely many of the subgroups $P_i$'s. Note that, because the intersection of finitely many parabolic subgroups is again a parabolic subgroup~\cite[Lemma 6.2]{PS23}, all elements of $\mathcal F$ are parabolic subgroups of $D$. Let $P_0=g_0 D_{Y_0} g_0^{-1}$ be an element of $\mathcal F$ with $\lvert Y_0\rvert$ minimal. We claim that $P_0=\bigcap_{i\in I}P_i$. To see this, we shall show that, for every $P\in\mathcal F$, $P_0\subseteq P$. Therefore, let $P=gD_{Y}g^{-1}\in\mathcal F$ and let us set $P'=P\cap P_0$, which is a parabolic subgroup of $D$. On the one side, by Theorem~\ref{thm:parabolicInsideParabolicDyer}, there exist $Z\subseteq Y_0$ and $h\in D_{Y_0}$ such that $P'=(g_0h)D_{Z}(g_0h)^{-1}$. In particular, $\lvert Z\rvert \le \lvert Y_0\rvert$. On the other side, $P'\in\mathcal F$ and hence $\lvert Z\rvert\ge\lvert Y_0\rvert$, by minimality of~$P_0$. It follows that $Z=Y_0$ and hence $P'=(g_0h) D_{Y_0} (g_0h)^{-1} =g_0 D_{Y_0} g_0^{-1} = P_0$. Hence $P_0=P'=P\cap P_0\subseteq P$, as we claimed.
\end{proof}

\begin{corollary}
    \label{cor:parabClosure}
    Let $(D,X)$ be a Dyer system and let $A\subseteq D$. There exists an inclusion-wise smallest parabolic subgroup of $D$ containing $A$.
\end{corollary}

\begin{proof}
    The claim follows as an application of Theorem~\ref{thm:intersectionParabSubgrp}, using the family $$\{P\subseteq D\mid \text{$P$ is a parabolic subgroup and }A\subseteq P\}.$$
\end{proof}

\begin{definition}[Parabolic closure]
    Let $(D,X)$ be a Dyer system; the \emph{parabolic closure} $\pc A$ of a subset $A\subseteq D$ is the inclusion-wise smallest parabolic subgroup of $D$ that contains $A$, as provided by Corollary~\ref{cor:parabClosure}.
\end{definition}

If $A=\{g\}$, then we shall write $\pc g$ instead of $\pc{\{g\}}$ for the sake of light notation, and call it the parabolic closure of the element $g$. Likewise, for a word $w$ in the alphabet~$X$, we denote by $\pc w$ the parabolic closure of the element $\overline w$.

\begin{lemma}
    \label{lem:parabClosureConj}
    Let $(D,X)$ be a Dyer system and let $h,g\in D$. Then $\pc{ghg^{-1}}=g\pc{h}g^{-1}$.
\end{lemma}
\begin{proof}
Let $P\subseteq D$ be a parabolic subgroup containing $ghg^{-1}$. Then $h\in g^{-1}Pg$, which is a parabolic subgroup, thus $\pc{h} \subseteq g^{-1}Pg$ and $g\pc{h}g^{-1} \subseteq P$. In particular, $g\pc{h}g^{-1} \subseteq \pc{ghg^{-1}}$. The other inclusion is symmetric.
\end{proof}

\section{Conjugating sets of parabolic subgroups}
\label{sec:ribbonProperty}

We now seek to answer the following question.

\begin{question}
Let $(D,X)$ be a Dyer system, let $Y,Y'\subseteq X$ and assume that the standard parabolic subgroups $D_Y,D_{Y'}$ are conjugate. Then which is the set of elements that conjugate $D_Y$ into $D_{Y'}$?
\end{question}
We will denote this set as
\[
\Conj(D_Y,D_{Y'}):=\{g\in D \mid gD_Y g^{-1}=D_{Y'} \}.
\]
In particular, when $Y$ and $Y'$ coincide, we will also describe the normaliser of a parabolic subgroup $D_Y$ in $D$, which we denote by
\[
\operatorname{N}_D(D_Y):=\{g\in D \mid gD_Y g^{-1}=D_{Y} \}.
\]
We first need to introduce some definitions. The following definition is a straight-forward generalization to Dyer groups of the definition for ribbons in Coxeter and Artin groups.

\begin{definition}
Let $(D,X)$ be a Dyer system with Dyer graph $\Gamma$, and let $Y\subseteq X$ such that~$D_Y$ is a standard parabolic subgroup of spherical type. Let $x\in Y$. Then $Y':=Y\setminus\{ x\}$ also has spherical type. We distinguish two cases. If $o(x)=2$, we consider the elements of maximal length in the Coxeter groups $D_{Y_2}$ and $D_{Y'_2}$. Let us denote them $w_{Y}$ and $w_{Y'}$ respectively. Then we have
\[
w_Y^{-1}w_{Y'} D_{Y'} w_{Y'}^{-1} w_Y= w_Y^{-1} D_{Y'} w_Y = D_{Z}
\]
for some subset $Z\subseteq Y$, since each longest element $w$ permutes the generators by conjugation. Notice that $Z_p\cup Z_\infty =Y'_p\cup Y'_\infty=Y_p\cup Y_\infty$ as they are point-wise fixed by the conjugation above. We call the element $w_{Y}^{-1}w_{Y'}$ and its inverse an \emph{elementary $(Y',Z)-$ribbon} and an \emph{elementary $(Z,Y')-$ribbon}, respectively.

If $o(x)>2$, then $x$ is in a cyclic component of $D_Y$ and we have that $xD_{Y'}x^{-1}=D_{Y'}$. In this case, we call the element $x$ and its inverse an \emph{elementary $(Y',Y')-$ribbon}.

For a general parabolic subgroup $D_Y$ not necessarily of spherical type, if there is $x\in X\setminus Y$ such that the connected component $\Gamma_U$ of $\Gamma_{Y\cup\{x\}}$ that contains $x$ is of spherical type, we call the following element and its inverse \emph{elementary ribbons}:
\[
r_{Y,x}:=\left\{\begin{array}{ll}
w_U^{-1} w_{U\setminus\{x\}}, & \text{if } o(x)=2, \\
    x, & \text{if } o(x)>2 \text{ (note that in this case $U=\{x\}$).}
\end{array}\right.
\]

We say that an element $r=r_1\dots r_q$ is a \emph{$(Y,Y')-$ribbon} if there is a sequence of sets of generators $Y=Y^1,\dots,Y^{q+1}=Y'$ such that each $r_i$ is an elementary $(Y^i,Y^{i+1})-$ribbon. The set of all $(Y,Y')-$ribbons is denoted by $\Ribb(Y,Y')$.
\end{definition}

\begin{remark}
    \label{rem:slideRibbon}
    It is clear from the definition of ribbon that, for all $Y,Y'\subseteq X$ and all $r\in\Ribb(Y,Y')$, $rD_Y=D_{Y'}r$. In particular, $\Ribb(Y,Y')\cdot D_Y=D_{Y'}\cdot \Ribb(Y,Y')$: we will often use this equality in the following of the section without commenting it.
\end{remark}

\begin{remark}\label{perpinribb}
For every $Y\subseteq X$, we have $Y^\perp\subseteq\Ribb(Y,Y)$. Indeed, take $x\in Y^{\perp}$, then $r_{Y,x}=x$.
In particular, if $\Gamma_Y$ is connected and there is $x\in Y$ with $o(x)>2$, then $\Ribb(Y,Y)=Y^\perp$. Also note that $\Ribb(Y,Y)\cap D_{p\cup\infty}=Y^\perp \cap D_{p\cup\infty}$.
\end{remark}

\begin{remark}
Notice that each edge in the graph $G$ that we defined in the previous section actually encodes a non trivial conjugation by an elementary ribbon. Indeed, let $(Y,t,t')$ be an edge of $G$, and let $Z=Y\setminus\{t'\}$ and $Z'=Y\setminus\{t\}$ be the two vertices that the edge joins. Then, we have that
\[
r_{Y,t'} D_{Z} r_{Y,t'}^{-1} =D_{Z'}.
\]
\end{remark}

\medskip

We should note that in \cite[Section 3.1]{K94}, a geometric approach, instead of our algebraic approach, is taken to study the conjugacy problem in Coxeter groups. We briefly describe this process for reference. To keep things as in the original, this is the only paragraph in this paper in which we write $g^{-1}\cdot g$ for conjugation instead of $g\cdot g^{-1}$.

Let $W$ be a Coxeter group with set of standard generators~$X$. The construction of the graph~$G$ from the previous section is given by the action of~$W$ on the simple roots $\Pi_X=\{\alpha_x\mid x\in X\}$ of its root system. In particular, the vertices are the subsets $Y\subseteq X$, while an edge joins two subsets $Y,Y'\subseteq X$ if there is $x\in X\setminus Y$ and an elementary ribbon~$r_{Y,x}$ such that $r_{Y,x}^{-1}\Pi_Y=\Pi_{Y'}$. It is then shown \cite[Theorem 3.1.3]{K94} that if there is $g\in W$ with $g^{-1}\Pi_Y=\Pi_{Y'}$ for some $Y,Y'\subseteq X$, then $Y$ and $Y'$ are in the same connected component of $G$, and $g$ can be written as a product of the elementary ribbons corresponding to the path in the graph joining $Y$ to $Y'$. Finally, it is proved \cite[Corollary 3.1.7]{K94} that if $g\in W$ is of minimal length in the coset $gW_{Y'}$ and $g^{-1}W_Yg=W_{Y'}$, then $g^{-1}\Pi_Y=\Pi_{Y'}$ and therefore $Y$ and $Y'$ are in the same connected component of $G$. 

\medskip

We will now state some definitions and results concerning Artin monoids, that we will use in the proof of our theorem.

\begin{definition}
	Let $(\Gamma, f, m)$ be a Dyer graph. The \emph{Artin monoid} $A^+ = A^+(\Gamma, m)$ associated with the Dyer graph $(\Gamma, f, m)$ is the monoid given by the following presentation
\[
		A^+ = \langle \ver{\Gamma} \mid \underbrace{uvu\cdots}_{m(u,v)\text{ terms}} = \underbrace{vuv\cdots}_{m(u,v)\text{ terms}} \text{ for all } \ \{u, v\} \text{ with } m(u,v)\neq\infty\rangle.
\]
\end{definition}

By \cite[Proposition 2.3]{brieskornsaito}, we know that Artin monoids satisfy right and left cancellation. Thanks to these, there are two partial orderings on the Artin monoid that we now introduce.

\begin{definition}
Let $A^+$ be an Artin monoid, and let $a,b\in A^+$. We define a partial order on $A^+$ by $a\preceq b$ if there exists $c\in A^+$ such that $b=ac$. In this case we say that $a$ is a \emph{left divisor} or a \emph{prefix} of $b$, and $b$ is a \emph{left multiple} of $a$. We call this order the \emph{prefix order} on $A^+$.
Analogously,  we define a partial order on $A^+$ by $a\succeq b$ if there exists $c\in A^+$ such that $b=ca$. In this case we say that $a$ is a \emph{right divisor} or a \emph{suffix} of $b$, and $b$ is a \emph{right multiple} of $a$. We call this order the \emph{suffix order} on $A^+$.
\end{definition}

These partial orderings extend from the Artin monoid to the Artin group defined with the same presentation. In the case of finite-type Artin groups, it provides the group with a \emph{Garside structure} ---we refer to \citep{brieskornsaito} for more details.

\medskip

Our last result describes the elements conjugating a standard parabolic subgroup into another (possibly equal) standard parabolic subgroup. This result is known as the \emph{ribbon property}.

\begin{theorem}[Ribbon property for Dyer groups]
    \label{thm:ribbonConj}
    Let $(D,X)$ be a Dyer system, and let $Y,Y'\subseteq X$ such that $D_Y$ and $D_{Y'}$ are conjugate. Then,
    \[
    \Conj(D_Y,D_{Y'})=\Ribb(Y,Y') \cdot D_Y.
    \]
\end{theorem}

\begin{proof}
It is obvious that $\Ribb(Y,Y') \cdot D_Y\subseteq\Conj(D_Y,D_{Y'})$, let us see the other inclusion. Let $g\in D$ be such that $gD_Yg^{-1}=D_{Y'}$. We choose a syllabic reduced word representing~$g$, which we denote $w$. We shall make some assumptions that will simplify our arguments. 

\begin{claim}
    \label{claim:noSuffixInDY}
    We may assume that neither $w$ nor any reduced equivalent word contains a suffix which represents an element of $D_Y$.
\end{claim}

\begin{proof}[Proof of Claim~\ref{claim:noSuffixInDY}]
    Let us assume that, up to passing to an equivalent reduced representative, $w$ can be written as $uv$, for some words $u$ and $v$ such that $v$ is non empty and $v$ is the maximal suffix satisfying $\overline{v}\in D_Y$. Then $D_{Y'}=g D_Y g^{-1}=\overline u D_Y \overline u^{-1}$, where no suffix of $u$ represents an element in $D_Y$. If $\overline{u}\in \Ribb(Y,Y') \cdot D_Y$, then $g\in \Ribb(Y,Y') \cdot D_Y$ as well, which is what we wanted. 
\end{proof}

\begin{claim}
    \label{claim:noSuffixInRibb}
    We may assume that neither $w$ nor any reduced equivalent word contains a suffix which represents an element of $\Ribb(Y,Z)$, for some $Z\subseteq X$.  
\end{claim}

\begin{proof}[Proof of Claim~\ref{claim:noSuffixInRibb}]
    Let us assume that, up to passing to an equivalent reduced representative, there exists $Z\subseteq X$ such that $w$ can be written as $uv$, for some words $u$ and $v$ such that $v$ is non empty and $\overline{v}\in \Ribb(Y,Z)$. Then $D_{Y'}=g D_Y g^{-1}=\overline u D_Z \overline u^{-1}$. If $\overline{u}\in \Ribb(Z,Y') \cdot D_Z$, then by Remark~\ref{rem:slideRibbon}  $$g\in \Ribb(Z,Y') \cdot D_Z \Ribb(Y,Z)= \Ribb(Y,Y') \cdot D_Y,  $$ which is what we wanted. 
\end{proof}

%We can also assume that $g$ has minimal length in the double coset $D_Y'gD_Y$ \maria{Explain why.}.

\begin{claim}
    \label{claim:minLenInDoubleCoset}
    We may assume that $g$ has minimal length in the double coset $D_{Y'}gD_Y$.
\end{claim}

\begin{proof}[Proof of Claim~\ref{claim:minLenInDoubleCoset}]
    Let us assume that there are $\alpha\in D_{Y'}$ and $\beta\in D_Y$ such that $g_0=\alpha g \beta$ has length shorter than $g$. Then $g_0$ conjugates $D_Y$ to $D_{Y'}$ as well. By the claim of the theorem for $g_0$, the latter decomposes as $rh$ for some $r\in\Ribb(Y,Y')$ and $h\in D_Y$. Then we have
    \[
    g=\alpha^{-1}g_0\beta^{-1}=\alpha^{-1}rh\beta^{-1}=r\alpha'h\beta^{-1}\in \Ribb(Y,Y')\cdot D_Y,
    \]
    where $\alpha'\in D_Y$ is such that $\alpha^{-1}r=r\alpha'$, as provided in Remark~\ref{rem:slideRibbon}.
\end{proof} 

We shall distinguish three cases.\medskip

\emph{Case $1$. Every generator $y\in Y$ has order $2$.} 

We first assume that the support of $g$ does not contain any letter of infinite order.  If $D_Y$ is finite, then the element $w_Y$ of maximum length in $D_Y$ exists. We know that $\Supp(w_Y)=Y$ and $\Supp(gw_Yg^{-1})=Y'$. Since $g$ has minimal length in the coset $gD_Y$, then $$|\Supp(g)\cup \Supp(w_Y)|> |\Supp(gw_Yg^{-1})|=|Y'|.$$ This means that there must be letters that cancel when reducing $ww_Yw^{-1}$ to a reduced representative. Since $ww_Y$ is reduced, then by Proposition~\ref{Proposition 2.8} we have that $l(gw_Yg^{-1})< l(gw_Y)+l(g)$ (with the usual word length). This is only possible if $g w_Y$ has a final letter~$t$ which cancels with~$t^{-1}$ at the beginning of~$g^{-1}$, so~$t$ is also a final letter of~$g$.

We will now use the theory of Artin monoids in our argumentation. By Property~$\mathcal{D}$, we can obtain a word finishing in $t$ using only elementary $M$-transformations of type II on the word $ww_Y$. This means that if we consider the word $ww_Y$ in the corresponding Artin monoid (for every letter of $ww_Y$, we choose its positive representative), we have a sequence of relations in the monoid that give us a word ending in $t$ if and only if we have the same sequence in the respective Coxeter group. We know that $g w_Y$ has as suffixes both $w_Y$ and $t$. Then it also has as suffix the least common multiple of the two, which is by definition $w_{Y\cup \{t\}}=r_{Y,t}w_Y$, so $g$ has $r_{Y,t}$ as suffix, contradicting Claim~\ref{claim:noSuffixInRibb} (this proof is analogous to \cite[Lemma 9]{Cumplido2019b}).\par

Let us now assume that $D_Y$ is not finite. For every final letter~$t$ of~$g$, we know by hypothesis that there is at least one $s\in Y$ such that $st\neq ts$. Take $Z\subset Y$ such that~$w_Z$ exists but $w_{Z\cup \{s\}}$ does not exist (notice that~$w_Z$ exists in the monoid if an only if~$D_Z$ is finite \cite[Theorem~5.6]{brieskornsaito}). Define $w_s:= w'w_Zs$ where $w'$ is the Coxeter element $s_1\cdots s_r$ of $Y\setminus(Z\cup \{s\})$. Again, by following the same argument of the previous case, we have that in the Artin monoid $gw_s$ and~$g$ share a final letter~$t$. Since $gw_s\succeq s$, we know that $gw_s$ has as suffix the least common multiple of~$s$ and~$t$, which is $\cdots sts$ of length at least $3$. Then $gw'w_Z\succeq t$ and then it has as suffix the least common multiple of~$w_Z$ and~$t$, which by definition is $w_{Z\cup\{t\}}$. Now, $gw'w_Zt^{-1} \succeq s$, and $w_{Z\cup\{t\}}t^{-1} \succeq w_Z$, which implies that $gw'w_Zt^{-1}\succeq w_{Z\cup \{s\}}$. But $w_{Z\cup \{s\}}$ does not exist, whence a contradiction.

\medskip

Now, if $\Supp(g)$ contains letters of infinite order, these letters only participate in commutation relations. This means that in the process of reducing $wvw^{-1}$ (where $v$ is~$w_Y$ or~$w_s$ in the previous cases depending on the finiteness of $D_Y$), the letters of infinite of order cancel each other. Let~$x$ be the rightmost infinite order letter of~$w$. Choose~$w$ to be such that~$x$ is as leftmost as possible and decompose $w=w_1xw_2$, so $l(wvw_2^{-1}x^{-1})< l(wv)+ l(w_2^{-1}x^{-1})$. This means that $x$ commutes with a reduced representative of $w_2vw_2^{-1}$.
We know by the previous cases that $l(w_2vw_2^{-1})=l(w_2v) +l(w_2)$ because otherwise we arrive to a contradiction. Then $w_2vw_2^{-1}$ is already reduced, and in particular~$x$ commutes with~$w_2$ and~$v$. As $\Supp(v)=Y$, this would imply that~$x$ is a final letter of~$g$ and a ribbon, which is a contradiction.  

\medskip

\emph{Case $2$. The subgraph $\Gamma_Y$ is connected and there exists one generator $y\in Y$ with order greater than $2$.} We note that in this case, necessarily $Y=Y'$ and what we want to show is that $g\in D_{Y^\perp}\times D_Y$. We write $g=g_1g_2g_3$ where $g_1,g_3\in D_Y$ and $g_2$ is of minimal syllabic length in the double coset $D_YgD_Y$. By \cite[Lemma 6.1]{PS23}, we have
\[
D_Y=D_Y\cap g_2D_Yg_2^{-1}=D_{Y\cap g_2Yg_2^{-1}}.
\]
Therefore, $Y=g_2Yg_2^{-1}$ and in particular by Claim \ref{claim:centraliser}, $y=g_2yg_2^{-1}$ and $g_2\in D_{\{y\}^\perp}$.

We will prove the result by proving that $g_2\in D_{Y^\perp}$ by induction on the cardinality of~$Y$. Note that $D_{Y^\perp}=\bigcap_{z\in Y}D_{\{z\}^\perp}$.

If $|Y|=1$, i.e. $Y=\{y\}$, then the reasoning above shows that we can write \mbox{$g=g_2(g_1g_3$)} where $g_2\in D_{Y^\perp}$ and $g_1g_3\in D_Y$, thus proving the result. Now assume that $Y=\{y,z_1,\dots, z_{m}\}$ and the result is true for cardinal $m$. Since $\Gamma_Y$ is connected, let us assume without loss of generality that $m(y,z_m)=\infty$. Therefore, $z_m\notin \Supp(g_2)$. This implies that $z_m=g_2 z_m g_2^{-1}$. Therefore, $Y\setminus \{z_m\}=g_2(Y\setminus \{z_m\})g_2^{-1}$. We distinguish two possibilities. If $o(z_m)>2$, then again by Claim \ref{claim:centraliser}, we have $g_2\in D_{\{z_m\}^\perp}$ and by induction hypothesis we obtain the result. If $o(z_m)=2$, then by Case 1 we have $g_2\in \Ribb(z_m,z_m)$, so $g_2$ can be written as a product of elementary ribbons $g_2=r_1\cdots r_k$. Assume that there is some $i\in\{1,\dots,k\}$ such that $r_i\notin \{z_m\}^\perp$. This implies that $z_m\in\Supp(g_2)$, which is a contradiction. Therefore, $g_2\in D_{\{z_m\}^\perp}$ and as before by induction hypothesis we obtain the result.

\medskip

\emph{Case $3$. The subgraph $\Gamma_Y$ has several connected components and there is at least one generator $y\in Y$ of order greater than $2$.}
It suffices to consider the case that $\Gamma_Y$ has two non-empty connected components. Let them be $\Gamma_A$ and $\Gamma_B$ with $A,B\subseteq Y$, so $D_Y=D_A\times D_B$. Since conjugation preserves the connected components, we have $\Gamma_{Y'}=\Gamma_{A'}\sqcup \Gamma_{B'}$ and $D_{Y'}=D_{A'}\times D_{B'}$, where $D_{A'}=gD_Ag^{-1}$ and $D_{B'}=gD_Bg^{-1}$. Let us assume that $y\in A$. Then by Claim \ref{claim:fixedcomponents}, $A=A'$ and by Cases 1 and 2 above, $g\in D_A\times D_{A^\perp}$ and $g\in\Ribb(B,B')\cdot D_B=D_{B'}\cdot\Ribb(B,B')$.

Let us write $g=g_1g_2$ with $g_1\in D_{B'}$ and $g_2\in \Ribb(B,B')$. Since every letter in~$A$ commutes with every letter in $B$, we have $D_{B'}\subseteq D_{A^\perp}$. Therefore, $g_2=g_1^{-1}g\in D_A\times D_{A^\perp}$. Let us write then $g_2=h_1h_2=h_2h_1$ with $h_1\in D_{A^\perp}$ and $h_2\in D_A$. Note that $D_A\subseteq\Ribb(B,B)$, again because every letter in $A$ commutes with every letter in $B$. Then, we have $h_1=g_2h_2^{-1}\in\Ribb(B,B')\cdot\Ribb(B,B)\subseteq\Ribb(B,B')$.

Now, we have $g=g_1g_2=(g_1h_2)h_1$ with $g_1h_2\in D_A\times D_{B'}=D_{Y'}$ and $h_1\in D_{A^\perp}\cap\Ribb(B,B')\subseteq\Ribb(Y,Y')$. This finishes the proof.
\end{proof}

When $Y=Y'$, we obtain an immediate corollary.

\begin{corollary}
    \label{cor:normaliserStdParabolic}
    Let $(D,X)$ be a Dyer system, and let $Y\subseteq X$. Then, the normaliser of~$D_Y$ in $D$ is
    \[
    \operatorname{N}_D(D_Y)=D_Y\rtimes \Ribb(Y,Y)
    \]
where the action of $\Ribb(Y,Y)$ on $D_Y$ is given by conjugation as described above.
\end{corollary}

\bibliography{bib}

\Addresses

\end{document}